\documentclass{article}

\usepackage{amsfonts}
\usepackage{amsmath}
\usepackage{a4wide}

\usepackage{amssymb}

\newtheorem{lem}{Lemma}[section]
\newtheorem{theo}[lem]{Theorem}
\newtheorem{coro}[lem]{Corollary}
\newtheorem{propo}[lem]{Proposition}
\newtheorem{rema}[lem]{Remark}
\newtheorem{defi}[lem]{Definition}

\newenvironment{lemma}
{\begin{lem}\sl } {\end{lem}}

\newenvironment{theorem}
{\begin{theo}\sl } {\end{theo}}

\newenvironment{corollary}
{\begin{coro}\sl } {\end{coro}}

\newenvironment{proposition}
{\begin{propo}\sl } {\end{propo}}

\newenvironment{remark}
{\begin{rema}\rm } {\end{rema}}

\newenvironment{proof}{\paragraph*{Proof}}
{\par}

\newcommand\qed{\hfill$\square$}

\newcommand\calE{{\mathcal E}}
\newcommand\calF{{\mathcal F}}

\newcommand\gal{{\mathrm{Gal}}}

\newcommand\GL{{\mathrm{GL}}}
\newcommand\SL{{\mathrm{SL}}}

\newcommand\eps\varepsilon
\newcommand\ph\varphi
\newcommand\C{{\mathbb C}}
\newcommand\F{{\mathbb F}}
\newcommand\Q{{\mathbb Q}}

\newcommand\Z{{\mathbb Z}}
\newcommand\R{{\mathbb R}}
\newcommand\height{{\mathrm h}}

\newcommand\HH{{\mathcal H}}
\newcommand\bfa{{\mathbf a}}

\newcommand\tilJ{{\widetilde J}}

\newcommand\tilg{{\tilde g}}

\newcommand\spl{{\mathrm{split}}}
\newcommand\spcar{{\mathrm{sp.Car.}}}
\newcommand\tors{{\mathrm{tors}}}
\newcommand\cl{{\mathrm{cl}}}

\title{Serre's Uniformity Problem in the Split Cartan Case}

\author{Yuri Bilu, Pierre Parent (Universit\'e de Bordeaux~I)}

1

\begin{document}

\maketitle

\begin{abstract}
We prove that there exists an integer~$p_{0}$ such that $X_{\mathrm{split}} (p)(\Q )$ is made 
of cusps and CM-points for any prime ${p>p_0}$. Equivalently, for any non-CM elliptic curve~$E$ 
over~$\Q$ and any prime ${p>p_0}$ the image of $\mathrm{Gal} (\bar{\Q} /\Q )$ by the representation 
induced by the Galois action on the $p$-division points of~$E$ is not contained in the normalizer 
of a split Cartan subgroup. This gives a partial answer to an old question of Serre.  
\medskip

\noindent
AMS 2000 Mathematics Subject Classification  11G18 (primary), 11G05, 11G16 (secondary). 

\end{abstract}

\section{Introduction}

Let~$N$ be a positive integer and~$G$ a subgroup of $\GL_{2} (\Z /N\Z )$ such that ${\det G=(\Z/N\Z)^\times}$. Then the corresponding modular 
curve~$X_G$, defined as a complex curve as ${\bar \HH/\Gamma}$, where~$\bar \HH$ is the extended Poincar\'e upper half-plane and~$\Gamma$ 
is the pullback of ${G\cap \SL_2(\Z/N\Z)}$ to $\SL_2(\Z)$, is actually defined over~$\Q$, that is, it has a geometrically integral $\Q$-model.  As usual, 
we denote by~$Y_G$ the finite part of~$X_G$ (that is,~$X_G$ deprived  of the cusps). The curve~$X_{G}$ has a natural (modular) model over~$\Z$ 
that we still denote by~$X_{G}$. The cusps define a closed subscheme of~$X_{G}$ 
over~$\Z$, and we define the relative curve~$Y_{G}$ over~$\Z$ as~$X_{G}$ deprived 
of the cusps. The 
set of integral points $Y_G(\Z)$ consists of those ${P\in Y_G(\Q)}$ for which 
${j(P)\in \Z}$, where~$j$ is, as usual, the modular invariant. 

In the special case when~$G$ is the normalizer of a split (or non-split) Cartan subgroup of  $\GL_{2}(\Z/N\Z )$, 
the curve~$X_G$ is denoted by $X_{\mathrm{split}}(N)$ (or $X_{\mathrm{nonsplit}}(N)$, respectively).
In this note we focus more precisely on the  case when~$G$ is the normalizer of a split 
Cartan subgroup of $\GL_{2}(\Z/p\Z )$ for $p$ a prime number, i.e.~$G$ is conjugate to the set of 
diagonal and anti-diagonal matrices $\mod p$, and we  prove the following theorem.

\begin{theorem}
\label{th1} 
There exists an absolute effective constant~$C$ such that for any prime number~$p$ and 
 any ${P\in Y_{\mathrm{split}}(p)(\Z)}$ we have ${\log|j(P)|\le 2\pi p^{1/2}+6\log p +C}$.
\end{theorem}

This is proved  in Section~\ref{sprel}  by a variation of the 
method of Runge, after some preparation in Section~\ref{sest} and~\ref{smod}. 
The terms $2\pi p^{1/2}$ and $6\log p$ seem to be optimal for the method. The 
constant~$C$ may probably be replaced by $o(1)$ when~$p$ tends to infinity. 

We apply Theorem~\ref{th1} to the arithmetic 
of elliptic curves. Serre proved~\cite{Se72} that for any elliptic curve~$E$ without complex multiplication 
(CM in the sequel), there exists  ${p_0(E)>0}$ such that for every prime ${p>p_0(E)}$ the natural Galois 
representation 
$$
\rho_{E,p}:\gal
(\bar{\Q} /\Q )\to \GL (E[p])\cong\GL_2(\Z/p\Z)
$$
is surjective. Masser and W\"ustholz~\cite{MW93},  Kraus~\cite{Kr97} and Pellarin~\cite{Pe01} gave effective versions of Serre's 
result; for more recent work, see, for instance, Cojocaru and Hall 
\cite{Co05,CH05}. 

Serre asked whether~$p_0$ can be made independent of~$E$:

\begin{quotation}{\sl
\noindent
does there exist an absolute constant~$p_0$ such that for any non-CM elliptic curve~$E$ 
over~$\Q$ and any prime ${p>p_0}$ the  Galois representation~$\rho_{E,p}$ is 
surjective? }
\end{quotation}
We refer to this as ``Serre's uniformity problem''. The general guess is that $p_{0}=37$ 
would probably do.  

The group $\GL_2(\Z/p\Z)$ has the following types of maximal proper subgroups: 
normalizers of (split and non-split) Cartan subgroups, Borel subgroups, and ``exceptional'' 
subgroups (those whose projective image is isomorphic to one of the groups~$A_4$,~$S_4$ 
or~$A_5$). To solve Serre's uniformity problem, one has to show that for sufficiently 
large~$p$, the image of the Galois representation is not contained in any of the above listed
maximal subgroups. (See~\cite[Section~2]{Ma76} for an excellent introduction into this topic.) Serre himself settled the case of exceptional 
subgroups, and the work  of Mazur~\cite{Ma78} on rational isogenies implies Serre 
uniformity for the Borel subgroups, so to solve Serre's problem we are left with the Cartan 
cases. Equivalently, one would like to prove that, for large~$p$, the only rational
points of the modular curves $X_{\mathrm{split}} (p)$ and $X_{\mathrm{nonsplit}} (p)$ 
are the cusps and CM points, in which case we will say that the rational points are 
\textsl{trivial}.

In the present note we solve  the split Cartan case of Serre's problem. 

\begin{theorem}
\label{tse}
There exists an absolute constant~$p_0$ such that for ${p>p_0}$ every point in $X_{\mathrm {split}}(p)(\Q )$ is either
a CM point or a cusp.
\end{theorem}

In other words, for any non-CM elliptic curve~$E$ 
over~$\Q$ and any prime ${p>p_0}$ the image of the Galois representation~$\rho_{E,p}$ is not contained in 
the normalizer of a split Cartan subgroup.

Several partial results in this direction were available before. In \cite{Pa05,Re08} it 
was proved, by very different techniques, that $X_{\mathrm{split}}(p)(\Q)$ is trivial for a 
(large) positive density of primes; but the methods of loc. cit. have failed to prevent a 
complementary set of primes from escaping them. In \cite{BP08} we allowed ourselves
to consider Cartan structures modulo higher powers of primes, and showed that, assuming the
Generalized Riemann Hypothesis, $X_{\mathrm{split}}(p^5)(\Q)$ is trivial for large 
enough $p$. 

Regarding possible generalizations, note that the Runge's method applies to the study of integral points on an affine curve~$Y$, defined over~$\Q$, if the following  \textsl{Runge condition} is satisfied:
$$
\text{$\gal(\bar\Q/\Q)$ acts non-transitively on the set ${X\setminus Y}$,}
\eqno(R)
$$
where~$X$ is the projectivization of~$Y$.    
The Runge     condition is satisfied for the curve
$X_{\mathrm{split}}(p)$ because it has two Galois orbits of cusps over~$\Q$. Runge's method 
also applies to other modular curves such as $X_0(p)$, but, unfortunately, it does not work (under the
form we use) with $X_{\mathrm {nonsplit}}(p)$, because all cusps of this curve 
are conjugate over~$\Q$ and the Runge condition fails. Moreover, we need
a weak version of Mazur's method to obtain integrality of rational points, and 
this is believed not to apply to $X_{\mathrm {nonsplit}}(p)$ (see \cite{Ch04}). Several other 
applications of our techniques are however possible, and at present we work on applying Runge's 
method to general modular curves over general number fields, see~\cite{BP08,BP10}.

\paragraph{Acknowledgments}
We thank Daniel Bertrand, Imin Chen, Henri Cohen,  Bas Edixhoven, Lo\"ic 
Merel, Joseph Oesterl\'e, Federico Pellarin,  Vinayak Vatsal and Yuri 
Zarhin for stimulating discussions and useful suggestions.

\paragraph{Convention}
Everywhere in this article the $O(\cdot)$-notation, as well as the Vinogradov notation ``$\ll$'' implies absolute effective constants.

\section{Siegel Functions}
\label{sest}

As above, we denote by~$\HH$ the Poincar\'e upper half-plane and put ${\bar\HH=\HH\cup\Q\cup\{i\infty\}}$. 
For ${\tau\in \HH}$ we, as usual,  put ${q=q_\tau=e^{2\pi i\tau}}$. For a rational number~$a$ we define ${q^a=e^{2\pi i a\tau}}$. 
Let ${\bfa=(a_1,a_2)\in \Q^2}$ be such that ${\bfa\notin \Z^2}$, and let ${g_\bfa:\HH\to \C}$ be the corresponding \textsl{Siegel function} 
\cite[Section~2.1]{KL81}. Then we have the following infinite product presentation for~$g_\bfa$ \cite[page~29]{KL81}:
\begin{equation}
\label{epga}
g_\bfa(\tau)= -q^{B_2(a_1)/2}e^{\pi ia_2(a_1-1)}\prod_{n=0}^\infty\left(1-q^{n+a_1}e^{2\pi ia_2}\right)\left(1-q^{n+1-a_1}e^{-2\pi i a_2}\right),
\end{equation}
where ${B_2(T)=T^2-T+1/6}$ is the second Bernoulli polynomial. 
We also have \cite[pages 27--30]{KL81} the relations
\begin{align}
\label{eperga}
g_\bfa\circ\gamma &=g_{\bfa\gamma} \cdot(\text{a root of unity}) \quad \text{for} \quad \gamma\in\SL_2(\Z),\\
\label{eaa'}
g_\bfa&=g_{\bfa'}\cdot(\text{a root of unity})
 \quad \text{when} \quad \bfa\equiv\bfa'\mod\Z^2.
\end{align}
Remark that the root of unity in~(\ref{eperga}) is of order dividing~$12$, and  in~(\ref{eaa'}) 
of order dividing $2N$, where~$N$ is the denominator of~$\bfa$ (the common denominator of~$a_1$ and~$a_2$). 
(For~(\ref{eperga}) use properties \textbf{K 0} and \textbf{K 1} of loc.\ cit., and for~(\ref{eaa'}) use \textbf{K 3} and the fact 
that $\Delta$ is modular of weight 12.)  
Moreover, 
\begin{equation}
g_\bfa\circ\gamma =g_{\bfa} \cdot(\text{a root of unity}) \quad \text{for} \quad \gamma\in\Gamma(N),
\end{equation}
the root of unity being of order dividing~$12N$, because $g_{\bfa}^{12N}$ is a modular function on 
$\Gamma (N)$ by Theorem~1.2 from \cite[page~31]{KL81}.

The following is immediate from~(\ref{epga}). 

\begin{proposition}
\label{pga}
Assume that ${0\le a_1<1}$. Then for ${\tau\in \HH}$ satisfying  ${|q_\tau|\le 0.1}$ we have
$$
\log \left|g_\bfa(\tau)\right|=\frac12B_2(a_1)\log|q|+ \log\left|1-q^{a_1}e^{2\pi ia_2}\right|+
\log\left|1-q^{1-a_1}e^{-2\pi ia_2}\right|+O(|q|)
$$
(where we recall that, all through this article, the notation $O(\cdot)$ as well as $\ll$ imply absolute 
effective constants).
\end{proposition}

For ${\bfa\in \Q^2\setminus\Z^2}$ Siegel's function~$g_\bfa$  is algebraic over the field $\C(j)$: this again follows from the fact that $g_\bfa^{12N}$ is $\Gamma(N)$-automorphic, where, as above,~$N$ is the denominator of~$\bfa$. Since~$g_\bfa$ 
is holomorphic and does not vanish on the upper half-plane~$\HH$ (again by Theorem~1.2 of loc.\ cit.), both~$g_\bfa$ and $g_\bfa^{-1}$  must be integral over the ring $\C[j]$. Actually, a stronger 
assertion holds.

\begin{proposition}
\label{psiu}
Both~$g_\bfa$ and   ${\left(1-\zeta_N\right)g_\bfa^{-1}}$ are integral over  $\Z[j]$. Here~$N$ is the
denominator of~$\bfa$  and~$\zeta_N$ is a primitive $N$-th root of unity. 
\end{proposition}

This  is, essentially, established in~\cite{KL81}, but is not stated explicitly therein. Therefore we briefly indicate the proof here. 
A holomorphic and $\Gamma(N)$-automorphic function ${f:\HH\to \C}$  admits the infinite $q$-expansion
\begin{equation}
\label{eexpan}
f(\tau)= \sum_{k\in \Z}a_kq^{k/N}.
\end{equation}
We call the $q$-series~(\ref{eexpan}) \textsl{algebraic integral} if the following two conditions 
are satisfied: the negative part of~(\ref{eexpan}) has only finitely 
many terms (that is, ${a_k=0}$ for large negative~$k$), and the  coefficients~$a_k$ are algebraic 
integers. Algebraic integral $q$-series form a ring. The invertible elements of this ring are 
$q$-series with invertible leading coefficient. By the \textsl{leading coefficient} of an 
algebraic integral $q$-series we mean~$a_m$, where  ${m\in \Z}$ is defined by ${a_m\ne 
0}$, but ${a_k=0}$ for ${k<m}$.

\begin{lemma}
\label{lint}
Let~$f$ be a $\Gamma(N)$-automorphic function  regular on~$\HH$ such that for every ${\gamma \in \Gamma
(1)}$ the $q$-expansion of ${f\circ \gamma}$ is algebraic integral. Then~$f$ is integral 
over $\Z[j]$. 
\end{lemma}

\begin{proof}
This is, essentially, Lemma~2.1 from \cite[Section~2.2]{KL81}. Since~$f$ is 
$\Gamma(N)$-automorphic, the set ${\{f\circ\gamma :\gamma\in \Gamma(1)\}}$ is finite. 
The coefficients of the polynomial ${F(T)=\prod(T-f\circ\gamma)}$ (where the product is 
taken over the finite set above)  are $\Gamma(1)$-automorphic functions  with algebraic 
integral $q$-expansions. Since they have no pole on~$\HH$, they belong to $\C[j]$ and even to
$\bar\Z[j]$, where~$\bar\Z$ is the ring of all algebraic integers, because the coefficients of their $q$-expansions are algebraic integers.  It follows that~$f$ is 
integral over~$\bar\Z[j]$, hence over $\Z[j]$. \qed
\end{proof}

\paragraph*{Proof of Proposition~\ref{psiu}}
The function $g_\bfa^{12N}$ is automorphic of level~$N$  and its 
$q$-expansion is algebraic integral (as one can easily see by transforming the infinite 
product~(\ref{epga}) into an infinite series).  By~(\ref{eperga}), the same is true for for 
every ${(g_\bfa\circ\gamma)^{12N}}$. Lemma~\ref{lint} now implies that $g_\bfa^{12N}$ is integral 
over $\Z[j]$, and so is~$g_\bfa$. 

Further, the  $q$-expansion of~$g_\bfa$ is invertible if ${a_1\notin \Z}$ and is 
${1-e^{\pm2\pi i a_2}}$ times an invertible $q$-series if ${a_1\in \Z}$. Hence the $q$-expansion 
of~$g_\bfa^{-1}$ is algebraic integral when ${a_1\notin \Z}$, and if ${a_1\in \Z}$ the 
same is true for ${\left(1-e^{\pm2\pi i a_2}\right)g_\bfa^{-1}}$. In the latter case~$N$ is the exact 
denominator of~$a_2$, which implies that ${(1-\zeta_N)/\left(1-e^{\pm2\pi i a_2}\right)}$ is 
an algebraic unit. Hence, in any case, ${(1-\zeta_N) g_\bfa^{-1}}$ has algebraic integral 
$q$-expansion, and the same is true with~$g_\bfa$ replaced by ${g_\bfa\circ\gamma}$ for
any ${\gamma\in \Gamma(1)}$ (we again use~(\ref{eperga}) and notice that~$\bfa$ and 
$\bfa\gamma$ have the same order in $(\Q/\Z)^2$). Applying Lemma~\ref{lint} to the function
${\left((1-\zeta_N) g_\bfa^{-1}\right)^{12N}}$,   we 
complete the proof.\qed

\section{A Modular Unit}
\label{smod}

In this section we define a special ``modular unit'' (in the spirit of \cite{KL81}) and study its asymptotic behavior at infinity. With the 
common abuse of speech, the modular invariant~$j$, as well as the other modular functions used below, may be viewed, depending on the context, 
as either automorphic functions on the Poincar\'e upper half-plane, or rational functions on the corresponding modular curves. 
 
Since the root of unity in~(\ref{eaa'}) are of order dividing ${2N}$, where~$N$ is a  denominator of~$\bfa$, the function 
$g_\bfa^{12N}$ will be well-defined  if we select~$\bfa$ in the set ${\left(N^{-1}\Z/\Z\right)^2}$.  
Thus,  fix a positive integer~$N$ and for a non-zero element~$\bfa$  of ${(N^{-1}\Z/\Z)^2}$ put ${u_\bfa=g_\bfa^{12N}}$. 
After fixing a choice for $\zeta_{N}$ in $\C$ (for instance $\zeta_{N} =e^{2i\pi /N}$), the analytic modular curve $X(N)(\C ):=
{\bar\HH}/\Gamma (N)$ has a modular model 
over $\Q (\zeta_{N})$, parameterizing isomorphism classes of generalized elliptic curves endowed with a basis of their $N$-torsion 
with determinant 1. As already noticed, the function~$u_\bfa$ is $\Gamma(N)$-automorphic and hence defines a rational function on the modular curve $X(N)(\C )$; in fact, it 
belongs to the field $\Q(\zeta_N)\bigl(X(N)\bigr)$. The Galois group of the latter field over $\Q(j)$  is isomorphic to 
$\GL_2(\Z/N\Z)/\{ \pm 1\}$, and we may identify the two groups to make the Galois action compatible with  the natural action 
of $\GL_2(\Z/N\Z)$ on $(N^{-1}\Z/\Z)^2$ in the following sense: for any ${\bar\sigma\in \gal\bigl(\Q\bigl(X(N)\bigr)\big/\Q(j)
\bigr)=\GL_2 (\Z/N\Z)} /\{ \pm 1\}$ and any non-zero ${\bfa\in (N^{-1}\Z/\Z)^2}$  we have ${u_\bfa^{\bar\sigma}=u_{\bfa\sigma}}$, where ${\sigma \in \GL_2(\Z/N\Z)}$ is a pull-back of~$\bar\sigma$. Notice that ${u_\bfa =u_{-\bfa}}$, which  follows from~(\ref{eperga}). 
For the proof of the statements above the reader may consult \cite[pp.\ 31--36]{KL81}, and especially Theorem~1.2,  Proposition~1.3 and the beginning of Section~2.2
therein.

From now on we assume that ${N=p\ge 3}$ is an odd prime number,  and that~$G$ is the normalizer of the diagonal subgroup of ${\GL_2(\F_p)}$. In this case there are two Galois orbits of cusps over~$\Q$, the first being the cusp at infinity, which is $\Q$-rational (we denote it by~$\infty$), and the second consisting of the  ${(p-1)/2}$ other cusps (denoted by ${P_1, \ldots, P_{(p-1)/2}}$),  which are defined over the real cyclotomic field $\Q(\zeta_p)^+$. According to the theorem of Manin-Drinfeld, there exists ${U\in \Q(X_G)}$ such that the principal divisor~$(U)$ is of the form ${m\bigl((p-1)/2\cdot \infty - (P_1+\cdots+P_{(p-1)/2})\bigr)}$ with some positive integer~$m$. Below we use Siegel's functions to find  such~$U$ explicitly with ${m=2p(p-1)}$.

\begin{remark}
\begin{enumerate}
\item
The general form of units we build is more ripe for generalization, but in the present case, using the 
$\Q$-isomorphism between $X_\spl (p)$ and $X_0 (p^2 )/w_p$, our unit could probably be expressed in terms 
of (products of) modular forms of shape $\Delta (nz)$.

\item
The assumption ${p\ge 3}$ is purely technical: the contents of this section extends, with insignificant changes, to ${p=2}$. 
\end{enumerate}
\end{remark}

Denote by $p^{-1}\F_p^\times$ the set of non-zero elements of ${p^{-1}\Z/\Z}$. Then the set 
$$
A= \bigl\{(a,0) : a\in p^{-1}\F_p^\times\bigr\}\cup \bigl\{(0,a) : a\in p^{-1}\F_p^\times\bigr\}
$$
is $G$-invariant. Hence the function 
$$
U=\prod_{\bfa\in A}u_\bfa
$$
belongs to the field $\Q(X_G)$. In particular, viewed as a function on~$\HH$, it is $\Gamma$-automorphic, where~$\Gamma$ is the pullback to $\Gamma(1)$ of ${G\cap \SL_2(\F_p)}$. 

More generally, for ${c\in \Z}$ put 
$$
\beta_c=\begin{pmatrix}1&0\\c&1\end{pmatrix}, \qquad U_c =U\circ \beta_c= \prod_{\bfa\in A\beta_c}u_\bfa
$$
(so that ${U=U_0}$).   

Let~$D$ be the familiar fundamental 
domain of $\SL_2(\Z)$ (that is, the hyperbolic triangle with vertices $e^{\pi i/3}$, 
$e^{2\pi i/3}$ and $i\infty$, together with the geodesic segments ${[i,e^{2\pi i/3}]}$ 
and ${[e^{2\pi i/3},i\infty]}$)  and ${D+\Z}$ the union of all translates of~$D$ by the rational integers. Recall also that~$j$ denotes the modular invariant.

\begin{lemma}
\label{lptau}
For any ${P\in Y_G(\C)}$ there exists ${c\in \Z}$ (in fact, even  ${c\in\{0, \ldots, (p-1)/2\}}$) and ${\tau\in D+\Z}$ such that ${j(\tau)=j(P)}$ and ${U_c(\tau)=U(P)}$. 
\end{lemma}

\begin{proof}
Let ${\tau'\in \HH}$ be such that ${j(\tau')=j(P)}$ and ${U(\tau')=U(P)}$. There exists ${\beta\in \Gamma(1)}$ such that ${\beta^{-1}(\tau') \in D}$. Now observe that the set 
${\left\{\beta_0, \ldots, \beta_{(p-1)/2}\right\}}$
is a full system of representatives of the double classes ${\Gamma\backslash\Gamma(1)/\Gamma_\infty}$, where~$\Gamma_\infty$ is the subgroup of $\Gamma(1)$ stabilizing~$\infty$.  Hence we may write ${\beta= \gamma\beta_c\kappa}$ with ${\gamma \in \Gamma}$, ${c\in\{0, \ldots, \lfloor p/2\rfloor\}}$ and ${\kappa\in \Gamma_\infty}$. Then ${\tau=\kappa\beta^{-1}(\tau')}$ is as desired. \qed
\end{proof}

\begin{proposition}
\label{pu}
For 
${\tau\in \HH}$ such that ${|q_\tau|\le 1/p}$  we have 
\begin{equation}
\label{euc1}
\bigl|\log \left|U_c(\tau)\right| - (p-1)^2\log |q_\tau|\bigr|\le 4\pi^2 \frac{p^2}{\log |q_\tau^{-1}|} + O(p\log p)
\end{equation}
if ${p\mid c}$, and 
\begin{equation}
\label{euc2}
\bigl|\log \left|U_c(\tau)\right| +2(p-1)\log |q_\tau|\bigr|\le 8\pi^2 \frac{p^2}{\log |q_\tau^{-1}|} +O(p)
\end{equation}
if ${p\nmid c}$.
\end{proposition}

For the proof of Proposition~\ref{pu}  we need an elementary, but crucial  lemma.

\begin{lemma}
\label{llogz}
Let~$z$ be a complex number, ${|z|<1}$, and~$N$ a positive integer. Then  
\begin{equation}
\label{eeta}
\left|\sum_{k=1}^N\log\bigl|1-z^k\bigr| \right|\le \frac{\pi^2}6\frac1{\log |z^{-1}|}+  O\left(1\right).
\end{equation}
\end{lemma}

\begin{proof}
We have 
${\bigl|\log|1+z|\bigr|\le -\log(1-|z|)}$
for ${|z|<1}$. 
Applying this with ${-z^k}$ instead of~$z$, we conclude that it
suffices to bound 
${-\sum_{k=1}^\infty\log(1-q^k)}$ with ${q=|z|}$. Since~(\ref{eeta}) is obvious for ${|z|\le 1/2}$, we may assume that \begin{equation}
\label{eqhalf}
1/2\le q<1.
\end{equation}
Put ${\tau = \log q/(2\pi i)}$. Then 
$$
-\sum_{k=1}^\infty\log|1-q^k| = \frac1{24}\log q -\log|\eta(\tau)|,
$$
where $\eta(\tau)$ is the Dedekind $\eta$-function. Since ${|\eta(\tau)|=|\tau|^{-1/2}|\eta(-\tau^{-1})|}$, we have 
\begin{equation}
\label{eqeta}
-\sum_{k=1}^\infty\log|1-q^k| = -\frac1{24}\log |Q|+\frac1{24}\log q+\frac12\log|\tau| -\sum_{k=1}^\infty\log|1-Q^k|
\end{equation}
with ${Q=e^{-2\pi i\tau^{-1}}= e^{4\pi^2/\log q}}$. The first term on the right of~(\ref{eqeta}) is exactly ${(\pi^2/6)/\log|z^{-1}|}$, the second term is negative, the third term is again negative (here we use~(\ref{eqhalf})), and the infinite sum is $O(1)$, again by~(\ref{eqhalf}). The lemma is proved.\qed
\end{proof}

\paragraph{Proof of Proposition~\ref{pu}}
Write ${q=q_\tau}$. Recall that for a rational number~$\alpha$ we define ${q^\alpha=e^{2\pi i\alpha\tau}}$. For ${a\in \Q/\Z}$ we denote by~$\tilde a$ the lifting of~$a$ to the interval $[0,1)$. 
Then 
for ${\tau\in \HH}$ satisfying  ${|q|\le 0.1}$ we deduce from Proposition~\ref{pga} that 
\begin{equation}
\label{ewa}
\begin{aligned}
\log \left|U_c(\tau)\right|=&&&\hskip-6pt6p\sum_{\bfa\in A\beta_c}B_2(\tilde a_1)\log|q|\\
&+&&\hskip-6pt 12p\sum_{\bfa\in A\beta_c}\Bigl( \log\bigl|1-q^{\tilde a_1}e^{2\pi ia_2}\bigr|+ \log\bigl|1-q^{1-\tilde a_1}e^{-2\pi ia_2}\bigr|\Bigr)+ O(p^2|q|).
\end{aligned}
\end{equation}
The rest of the proof splits into two cases and relies on the identity
$$
\sum_{k=1}^{N-1}B_2\left(\frac kN\right) =-\frac{(N-1)}{6N}.
$$

\paragraph{The first case: ${p\mid c}$}
In this case ${A\beta_c=A}$. Hence 
\begin{equation}
\label{ebern}
\sum_{\bfa\in A\beta_c}B_2(\tilde a_1)= \sum_{k=1}^{p-1}B_1\left(\frac kp\right)+ (p-1)B_2(0)= \frac{(p-1)^2}{6p}.
\end{equation}
Further,
\begin{equation}
\label{esumlo}
\sum_{\bfa\in A\beta_c}\Bigl( \log\bigl|1-q^{\tilde a_1}e^{2\pi ia_2}\bigr|+ \log\bigl|1-q^{1-\tilde a_1}e^{-2\pi ia_2}\bigr|\Bigr)= 2\sum_{k=1}^{p-1}\log\bigl|1-q^{k/p}\bigr| +\log\left|\frac{1-q^p}{1-q}\right|+\log p.
\end{equation}
Lemma~\ref{llogz} with ${z=q^{1/p}}$ implies that 
$$
\sum_{k=1}^{p-1}\log\bigl|1-q^{k/p}\bigr| \le \frac{\pi^2}6\frac p{\log |q^{-1}|}+  O\left(1\right).
$$
 Also, 
${\log\left|1-q^p\right|\ll |q|^p}$ and  
${\log\left|1-q\right|\ll |q|}$. 
Combining all this with~(\ref{ewa}),~(\ref{ebern}) and~(\ref{esumlo}), we obtain~(\ref{euc1}).

\paragraph{The second case: ${p\nmid c}$} 
In this case 
$$
A\beta_c=\{(a,0) : a\in p^{-1}\F_p^\times\}\cup\{(a,ab) : a\in p^{-1}\F_p^\times\},
$$ 
where ${b\in \Z}$ satisfies ${bc\equiv 1\bmod p}$. Hence 
$$
\sum_{\bfa\in A\beta_c}B_2(\tilde a_1)= 2\sum_{k=1}^{p-1}B_1\left(\frac kp\right)=-\frac{p-1}{3p}. 
$$
Further,
$$
\sum_{\bfa\in A\beta_c}\Bigl( \log\bigl|1-q^{\tilde a_1}e^{2\pi ia_2}\bigr|+ \log\bigl|1-q^{1-\tilde a_1}e^{-2\pi ia_2}\bigr|\Bigr)= 2\sum_{k=1}^{p-1}\log\bigl|1-q^{k/p}\bigr|+ 2\sum_{k=1}^{p-1}\log\bigl|1-(q^{1/p}e^{2\pi ib/p})^k\bigr|.
$$
Again using Lemma~\ref{llogz}, we  complete the proof.   \qed

\section{Proof of Theorem~\ref{th1}}
\label{sprel}

In this section~$p$ is a prime number and~$G$ is the normalizer of the diagonal subgroup 
of $\GL_2(\Z/p\Z)$. Define the ``modular units''~$U_c$ as in Section~\ref{smod}. Recall 
that ${U=U_0}$ belongs to the field $\Q(X_G)$. 
Theorem~\ref{th1} is a consequence of the following two statements. 

\begin{proposition}
\label{pana}
Assume that ${p\ge 3}$. For any ${P\in Y_G(\C)}$  we have either 
$$
\log|j(P)|\le 2\pi p^{1/2}+6\log p+O(1)
$$
or 
\begin{equation}
\label{ean}
\log |j(P)| \le \frac1{2(p-1)}\bigl|\log |U(P)|\bigr| +2\pi p^{1/2}-6\log p+O\bigl(1\bigr). 
\end{equation}
\end{proposition}

\begin{proposition} 
\label{pari}
For  ${P\in Y_G(\Z)}$ we have ${0\le \log|U(P)| \le 24 p\log p}$.
\end{proposition}
Combining the two propositions, we find that for ${P\in Y_{\mathrm{split}}(p)(\Z)}$ we have
$$
\log |j(P)| \le 2\pi p^{1/2} +6\log p +O(1),
$$
which proves Theorem~\ref{th1} for ${p\ge 3}$. 

A similar approach can be used for ${p=2}$ as well, but in this case it is easier to appeal to the general Runge theorem: if an affine curve~$Y$, defined over~$\Q$, has $2$ (or more) rational points at infinity, then integral points on~$Y$ are effectively bounded; see, for instance, \cite{Bo83,Le08}.

\paragraph{Proof of Proposition~\ref{pana}}
According to Lemma~\ref{lptau}, there exists ${\tau\in D+\Z}$ and ${c\in \Z}$ with ${U_c(\tau)=U(P)}$ and ${j(\tau)=j(P)}$. We write ${q=q_\tau}$. Since ${\tau\in D+\Z}$, we have 
\begin{equation}
\label{ejtau}
j(\tau)=q^{-1}+O(1),
\end{equation} 
which implies that either  ${\log|j(P)|\le 2\pi p^{1/2}+6\log p+O(1)}$ or ${\log|q^{-1}| \ge 2\pi p^{1/2}+6\log p}$. In the latter case we apply Proposition~\ref{pu}. When 
 ${p\nmid c}$ it yields
$$
\left|\log|q| + \frac1{2(p-1)} \log |U_c(\tau)|\right|\le  \frac{8\pi^2 p^2}{2(p-1)(2\pi p^{1/2}+6\log p)}+O(1)= 2\pi p^{1/2}-6\log p+O(1),
$$
which, together with~(\ref{ejtau}), implies the result. In the case ${p\mid c}$ Proposition~\ref{pu} gives
$$
\left|\log|q| - \frac1{(p-1)^2} \log |U_c(\tau)|\right|\le  \frac{4\pi^2 p^2}{(p-1)^2(2\pi p^{1/2}+6\log p)}+O(1)= O(1),
$$
which implies even a better bound than needed.
\qed

\paragraph{Proof of Proposition~\ref{pari}}
Since~$U$ belongs to $\Q(X_G)$ and has no pole or zero outside the cusps, $U(P)$ is a non-zero rational number.
Let ${\zeta=\zeta_p}$ be a primitive $p$-th root of unity.  Since~$U$ is a product of ${24p(p-1)}$ Siegel functions, Proposition~\ref{psiu} implies that both~$U$ and ${(1-\zeta)^{24p(p-1)}U^{-1}}$ are integral over $\Z[j]$. Hence, for ${P\in Y_G(\Z)}$ both the numbers ${U(P)}$ and ${(1-\zeta)^{24p(p-1)}U(P)^{-1}}$ are algebraic integers. Since ${U(P)\in \Q^\times}$, it is a non-zero rational integer; in particular, ${\log|U(P)|\ge 0}$. 
Further, $U(P)$ divides ${(1-\zeta)^{24p(p-1)}}$. Taking the $\Q(\zeta)/\Q$-norm, we 
see that ${U(P)^{p-1}}$ divides ${p^{24p(p-1)}}$. This proves the proposition. \qed

\section{Proof of Theorem~\ref{tse}}

First of all, recall the following \textit{integrality property} of the $j$-invariant. 

\begin{theorem} {\bf (Mazur, Momose, Merel).}
\label{tmmm1}
For a prime ${p= 11}$ or ${p\geq 17}$, the $j$-invariant $j(P)$ of any non-cuspidal 
point of $X_\spl(p)(\Q )$ belongs to~$\Z$.
\end{theorem}

This is a combination of results of Mazur~\cite{Ma78},  Momose \cite{Mo84} and  Merel \cite{Me05}. 
For more details see the Appendix (Section~\ref{sapp}), where we give a short 
unified proof. 

  Denote by $\height (\alpha)$ the absolute logarithmic height of an algebraic number~$\alpha$. 
If~$\alpha$ is a non-zero rational integer, then ${\height(\alpha)=\log|\alpha|}$; it follows in 
particular from the above that, if $E$ is an elliptic curve over $\Q$ endowed with a normalizer of split Cartan 
mod $p$ with $p\ge 17$, then $\height (j_{E})=\log |j_{E}|$.

In view of Theorem~\ref{tmmm1}, Theorem~\ref{tse} is a straightforward consequence of  Theorem~\ref{th1}  and the following proposition, whose proof will be the goal of 
this section.

\begin{proposition}
\label{pubo} 
There exists an absolute effective constant~$\kappa$ such that the following holds. Let~$p$ 
be a prime number, and~$E$ a non-CM elliptic curve over~$\Q$, endowed with a structure 
of normalizer of split Cartan subgroup in level~$p$. Then  
\begin{equation}
\label{enogrh}
\height(j_E) =\log |j_{E} |\ge \kappa p.
\end{equation}
\end{proposition}

The proof of Proposition~\ref{pubo} relies on Pellarin's refinement~\cite{Pe01} of the   Masser-W\"ustholz famous bound~\cite{MW90} for  the degree of the minimal isogeny between two isogenous elliptic 
curves. 

\begin{theorem}
\label{tpel}
{\bf (Masser-W\"ustholz, Pellarin)} Let~$E$ be an elliptic curve defined over a number 
field~$K$ of degree~$d$. Let~$E'$ be another elliptic curve, defined over~$K$ and isogenous
to~$E$. Then there exists an isogeny  ${\psi:E\to E'}$ of degree at most ${\kappa(d)
\left(1+\height(j_E)\right)^2}$, where the constant $\kappa(d)$ depends only on~$d$ 
and is effective.
\end{theorem}
Masser and W\"ustholz had exponent~$4$, and Pellarin reduced it to~$2$, which is crucial for us (in fact, any exponent below~$4$ would do).  He also gave an 
explicit expression for~$\kappa(d)$.

\begin{corollary}
\label{cpel}
Let~$E$ be a non-CM elliptic curve  defined over a number field~$K$ of degree~$d$, and 
admitting a cyclic isogeny over~$K$ of degree $\delta$. Then $\delta\le {\kappa(d)
\left(1+\height(j_E)\right)^2}$.
\end{corollary}
\begin{proof}
Let $\phi$ be a cyclic isogeny from~$E$ to~$E'$, and let ${\phi^D \colon E' \to E}$ be
the dual isogeny. Let ${\psi \colon E\to E'}$ be an isogeny of degree bounded by 
${\kappa(d)\left(1+\height(j_E)\right)^2}$; without loss of generality,~$\psi$ may be
assumed cyclic. As $E$ has no CM, the composed map ${\phi^D \circ \psi}$ must be 
multiplication by some integer, so that ${\phi =\pm \psi}$. 
\qed
\end{proof}

\paragraph{Proof of Proposition~\ref{pubo}}

For an elliptic curve~$E$ endowed with a structure of normalizer of split 
Cartan subgroup in level~$p$ over~$\Q$, write~$C_{1}$ and~$C_{2}$ for the
obvious two independent $p$-isogenies defined over the quadratic field~$K$ 
cut-up by the inclusion of the Cartan group mod~$p$ in its normalizer. Set ${E_i 
:=E/C_i}$  and recall that there is a cyclic $p^2$-isogeny over~$K$ from~$E_{1}$ to~$E_{2}$, factorizing as a product of two $p$-isogenies: 
$$
\varphi \colon E_{1}\to E\to E_{2} .
$$
It follows from  Corollary~\ref{cpel} that ${\height(j_{E_{i}}) \geq 
\kappa_{1} p}$ for ${i=1,2}$. 

A result of Faltings \cite[Lemma 5]{Fa84} now asserts that ${\height_{\calF} (E_{1}) \leq 
\height_{\calF} (E) +\frac{1}{2}\log p}$, where $\height_{\calF}$ is Faltings' 
semistable height\footnote{Recall that $h_{\calF}(E)$ is defined as ${[K_{1} :\Q ]^{-1} \deg \omega }$, where~$K_{1}$ is an extension of~$K$ such that~$E$ has semi-stable 
reduction at every place of~$K_1$, and~$\omega$ is a 
N\'eron differential on~$E\vert_{K_1}$; it is independent of the choice of~$K_1$ and~$\omega$.}. Finally,   for any elliptic curve~$\calE$ over  a number field we have
$$
\bigl|\height (j_{\calE})-12\height_{\calF} (\calE )\bigr|\leq 6\log \bigl(1+\height 
(j_{\calE})\bigr)+O(1),
$$
see \cite[Proposition~2.1]{Si84}. (Pellarin shows that $O(1)$ can be replaced by 47.15, see~\cite{Pe01}, equation~(51) on 
page 240.) This completes the proof of Proposition~\ref{pubo} and of Theorem~\ref{tse}.
\qed

\section{Appendix: Integrality of the $j$-Invariant}
\label{sapp}

This appendix is mainly of expository nature: following  Mazur, Momose and Merel, we show
that rational points on $X_{\mathrm{split}}(p)$ are, in fact, integral.  

\begin{theorem} {\bf (Mazur, Momose, Merel).}
\label{tmmm}
For a prime ${p= 11}$ or ${p\geq 17}$, the $j$-invariant $j(P)$ of any non-cuspidal 
point of $X_\spl(p)(\Q )$ belongs to~$\Z$.
\end{theorem}
The proof of this theorem is somehow scattered in the literature. Mazur~\cite[Corollary 
4.8]{Ma78} proved that a prime divisor~$\ell$ of the denominator of $j(P)$ must either be~$2$, 
or~$p$, or satisfy ${\ell\equiv\pm 1\mod p}$. The cases ${\ell\equiv\pm 1\mod p}$ and 
${\ell =p}$ were settled by Momose \cite[Proposition 3.1]{Mo84}, together with the case 
${\ell =2}$ when ${p\equiv 1\mod 8}$ \cite[Corollary 3.6]{Mo84}. Finally the case 
${\ell =2}$ with ${p\not\equiv 1\mod 8}$ was treated by Merel \cite[Theorem 5]{Me05}. 
The aim of this appendix is therefore to present a short unified proof, for the reader's convenience.
For simplicity, we actually assume in all what follows that $p\geq 37$.

Recall that the curve $X_\spl(p)$ parametrizes (isomorphism classes of) elliptic curves endowed with 
an \textsl{unordered} pair of independent $p$-isogenies. Let ${P=\bigl(E,\{A,B\}
\bigr)}$ be a $\Q$-point on $X_\spl(p)$, which we may assume to be non CM. Then the 
isogenies~$A$ and~$B$ are defined over a number field~$K$ with degree at most 2.
\begin{proposition}
\label{peab}
Let ${P=\bigl(E,\{A,B\}\bigr)\in X_\spl(p)(\Q)}$ and~$K$ be defined as above. Let
${\cal O}_K$ be its ring of integers. Then we have the following.
\begin{enumerate}
\item
\label{insup}
The curve~$E$ is not potentially supersingular at~$p$. 
\item
\label{ipfiber}
The points $(E,B)$ and $(E/A, E[p]/A)={(E/A, A^\ast)}$, where~$A^\ast$ is the isogeny 
dual to~$A$, coincide in the fibers of characteristic $p$ of $X_0(p)_{/{\cal O}_K}$.
\item
\label{ipsplit}
The field~$K$ is quadratic over~$\Q$, and~$p$ splits in~$K$. 
\end{enumerate}
\end{proposition}

\begin{proof}
Parts~(\ref{insup}) and~(\ref{ipsplit}) are proved in~\cite{Mo84}, Lemmas~1.3
and~3.2 respectively. Part~(\ref{ipfiber}) follows from \cite[proof of 
Proposition~3.1]{Pa05}. For the convenience of the reader we here sketch a proof (with 
somewhat different (and simpler) arguments).

  It follows from Serre's study of the action of inertia groups $I_{p}$ at $p$ on the formal 
group of elliptic curves that if $E$ is potentially supersingular then $I_{p}$ (potentially) 
acts via a ``fundamental character of level 2", so that the image of inertia contains a subgroup
of order 4 or 6 in a nonsplit Cartan subgroup of $\GL (E[p])$ (see \cite[Paragraph 1]{Se72}). 
This gives a contradiction 
with the fact that a subgroup of index 2 in the absolute Galois group of $\Q$ preserves two 
lines in $E[p]$, whence part~(\ref{insup}). 

   For~(\ref{ipfiber}) we remark that we may assume the schematic closure of $A$ to be 
\'etale over $\cal O$ (the ring of integers of a completion $K_{\cal P}$ of $K$ at a prime 
$\cal P$ above $p$, whose residue field we denote by $k_{\cal P}$): indeed, as $E$ is not 
potentially supersingular at $\cal P$, at most one line in $E[p]$ can be purely radicial over 
$k_{\cal P}$. Up to replacing $K_{\cal P}$ by a finite ramified extension, we shall also
assume $E$ is semistable over $K_{\cal P}$. Now $E/A$ is isomorphic over 
${\overline{k}}_{\cal P}$ to $E^{(p)}$ via the Vershiebung isogeny, and the latter is in 
turn isomorphic to $E_{/k_{\cal P}}$ as $E$ has a model over $\Z$. Moreover the
isomorphism between $B$ and $E[p]/A$ as $K$-group schemes induced by the projection $E
\to E/A$ extends to an isomorphism over $\cal O$ by Raynaud's theorem on group schemes
of type $(p,\dots ,p)$, as recalled in \cite[Proof of Lemma 1.3]{Mo84}. It follows that 
$(E,B)_{k_{\cal P}}$ is isomorphic to $(E/A ,E[p]/A)_{k_{\cal P}} =(w_{p} (E,A)
)_{k_{\cal P}}$, whence~(\ref{ipfiber}). 

   Pushing this reasonning further yields~(\ref{ipsplit}). We indeed first note that, by 
 Mazur's theorem on rational isogenies, $K\neq\Q$ for $p\geq 37$, and with the above 
notations $(E,B)$ is also equal to $\sigma (E,A)$ for a $\sigma\in\mathrm{Gal}(K/\Q )$. 
Now $(E,A)$ and $(E,B)$ define two points in $X_0(p)({\cal O})$. The two components of 
$X_{0} (p)_{\F_p}$ are left stable by ${\mathrm{Gal}}(\overline{\F}_p /\F_p )$, but 
switched by $w_{p}$. Therefore if $p$ does not split in $K$, the fact that $\sigma (E,A)_{
k_{\cal P}} =(E,B)_{k_{\cal P}}=(w_{p} (E,A))_{k_{\cal P}}$ leaves no other choice
to $(E,A)$ than specializing at the intersection of the components. As those intersection 
points modularly correspond to supersingular curves, this contradicts~(\ref{insup}), and
completes the proof. \qed
\end{proof}

\bigskip

The curve $X_\spl(p)$  admits an obvious double covering by the curve $X_\spcar (p)$, 
parameterizing elliptic curves endowed with an \textsl{ordered} pair of $p$-isogenies. We 
denote by~$w$ be the generator of the Galois group of this covering, that is,~$w$ modularly 
exchanges the two $p$-isogenies. In symbols, if  $\bigl(E,(A,B)\bigr)$ is a point on 
$X_\spcar(p)$, then  ${w\bigl(E,(A,B)\bigr)=\bigl(E,(B,A)\bigr)}$. We recall certain
properties of the modular Jacobian $J_0(p)$ and its \textsl{Eisenstein quotient} $\tilJ(p)$ 
(see~\cite{Ma77}). 

\begin{sloppypar}

\begin{proposition}
\label{pmaz}
Let~$p$ be a prime number. Then we have the following.

\begin{enumerate}

\item \cite[Theorem~1]{Ma77}\quad 
The group $J_0(p)(\Q)_\tors$ is cyclic and generated by $\cl(0-\infty)$, where~$0$ 
and~$\infty$ are the cusps of $X_0(p)$. Its order is equal to the numerator of the quotient
${(p-1)/12}$.

\item \cite[Theorem~4]{Ma77}\quad 
The group $\tilde J(p)(\Q)$ is finite. Moreover, the natural projection ${J_0(p)\to \tilde 
J(p)}$ defines an isomorphism ${J_0(p)(\Q)_\tors \to \tilJ(p)(\Q)}$.

\end{enumerate}

\end{proposition} 
As Mazur remarks, Raynaud's theorem on group schemes of type $(p,\ldots ,p)$ 
insures that $J_0(p)(\Q)_\tors$ defines a $\Z$-group scheme which, being constant in the 
generic fiber, is \'etale outside $2$, and which at 2 has \'etale quotient of rank at least 
half that of $J_0(p)(\Q)_\tors$.
\end{sloppypar}

\paragraph{Proof of Theorem~\ref{tmmm}}

For an element~$t$ in the $\Z$-Hecke 
algebra for $\Gamma_{0} (p)$, define the morphism~$g_t$ from 
$X_\spcar^{\mathrm{smooth}} (p)_{/\Z}$ to $J_{0} (p)_{/\Z}$ which 
extends the morphism on generic fibers:
$$
g_{t} \colon
\left\{
\begin{array}{rcl}
X_{\mathrm{sp.Car.}} (p) & \to & J_{0} (p) \\
Q=\bigl(E,(A,B)\bigr) & \mapsto & t\cdot {\mathrm{cl}} \bigl((E,A) -
(E/B ,E[p]/B )\bigr) .
\end{array}
\right.
$$
Let ${J_{0} (p)\stackrel\pi\to \tilJ(p)}$ be the projection to the Eisenstein quotient, 
and ${\tilg_t :=\pi\circ g_t}$. One checks that ${g_t \circ w=-w_{p} \circ g_t}$ 
 and one knows that 
$(1+w_{p})$ acts trivially on $\tilJ(p)$ from~\cite[Proposition 17.10]{Ma77}. 
Therefore $\tilg_t$ actually factorizes through a $\Q$-morphism from $X_\spl (p)$ to 
$\tilJ(p)$, that we extend by the universal property of N\'eron models to a map from 
$X^{\mathrm{smooth}}_\spl (p)_{/\Z}$ to $\tilJ(p)_{/\Z}$. We still denote this 
morphism by~$\tilg_t$ and we put ${\tilg=\tilg_1}$.

Let~$P$ be a rational point on $X_\spl(p)$, and~$\ell$ a prime divisor of the denominator 
of $j(P)$.  Then~$P$ specializes to a cusp at~$\ell$. Recall that $X_\spl(p)$ has one cusp 
defined over~$\Q$ (the \textsl{rational cusp}), and ${(p-1)/2}$ other cusps, conjugate 
over~$\Q$. We first claim that~$P$ specializes to the rational cusp. Indeed, it follows from
Propositions~\ref{peab}~(\ref{ipfiber}) that ${\tilg(P)(\F_p )=0(\F_p )}$, and by the 
remark after Proposition~\ref{pmaz}, ${\tilg(P)(\Q )=0(\Q )}$ (recall $p>2$). The non-rational cusps of
$X_{\mathrm{split}} (p)(\C )$ map to $\mathrm{cl} (0-\infty)$ in $J_{0} (p)(\C )$ (this 
can be seen with the above modular interpretation of $\tilg_t$, using the fact that the 
non-rational cusps specialize at $p$ to a generalized 
elliptic curve endowed with a pair of {\it \'etale} isogenies. Or, if $f$ denotes the 
map $f\colon X_{\mathrm{sp.C.}} (p)\to X_0 (p)$, $(E,(A,B))\mapsto (E,A)$, one has
$g_1 = \mathrm{cl} (f-w_p fw)$, and as $f(c_{i})=0\in X_0 (p)$ for $c_i$ a non-rational
cusp and $w$ permutes the $c_{i}$s, one sees that $\tilg (c_{i})=\mathrm{cl} (0-
\infty)$. For more details see for instance the proof of Proposition 2.5 in \cite{Mo84}). 
Therefore, for $p\geq 11$ and $p\neq13$, Proposition~\ref{pmaz} implies that if $P$ 
specializes to a non-rational cusp at $\ell$ then $\tilg (P)$ would not be $0$ at $\ell$, a 
contradiction. 

Now take an $\ell$-adically maximal element~$t$ in the Hecke algebra which kills the 
winding ideal~$I_{e}$. Again, as $t(1+w_{p}) =0$, the above morphism $g_{t}$ factorizes
through a morphism $g_{t}^+$ from $X_{\mathrm{split}}^{\mathrm{smooth}} (p)_{/\Z}$
to $t\cdot J_{0} (p)_{/\Z}$. Moreover $g_{t}^+ (P)$ belongs to $t\cdot J_{0} (p)(\Q )$, 
hence is a torsion point, as $t\cdot J_{0} (p)$ is isogenous to a quotient of the winding 
quotient of $J_{0} (p)$. We see as above by looking at the fiber at $p$ that $g_{t}^+ (P)=0$ 
at $p$, hence generically. We then easily check by using the $q$-expansion principle, as in 
\cite[Theorem 5]{Me05}, that $g_{t}^+$ is a formal immersion at the specialization 
$\infty (\F_\ell )$ of the rational cusp on $X_{\mathrm{split}} (p)$. This allows us 
to apply the classical argument of Mazur (see e.g. \cite[proof of Corollary 4.3]{Ma78}), 
yielding a contradiction: therefore $P$ is not cuspidal at $\ell$.  \qed

\end{document}